\documentclass[a4paper,12pt]{amsart}

\usepackage{amssymb,amsthm,amsmath} 
\usepackage[T1]{fontenc}
\usepackage[utf8]{inputenc}  
\usepackage{a4wide}
\usepackage{enumerate}     	
\usepackage{graphicx}
\usepackage{bm} 
\usepackage{tikz}
\usepackage{verbatim}

\usetikzlibrary{arrows,positioning} 
\tikzset{
    >=stealth',
    punkt/.style={
           rectangle,
           rounded corners,
           draw=black, very thick,
           text width=6.5em,
           minimum height=2em,
           text centered},
    pil/.style={
           ->,
           thick,
           shorten <=2pt,
           shorten >=2pt,}
}

\theoremstyle{definition}
 
\newtheorem{proposition}{Proposition}[]
\newtheorem{lemma}{Lemma}[]

\newtheorem{remark}{Remark}[] 
\newtheorem{theorem}[]{Theorem}
\newtheorem{assumption}{Assumption}[]

\title[asymptotic relations between control problems]{On asymptotic relations between singular and constrained control problems of one-dimensional diffusions}
\date{\today}
\vspace{5pt}
\author{Jukka Lempa}
\address{Jukka Lempa, Department of Mathematics and Statistics, University of Turku, FI - 20014 Turun Yliopisto, Finland, \texttt{jumile@utu.fi}}
\author{Harto Saarinen}
\address{Harto Saarinen, Department of Mathematics and Statistics, University of Turku, FI - 20014 Turun Yliopisto, Finland, \texttt{hoasaa@utu.fi}}

\begin{document}

\begingroup
\let\newpage\relax%
\maketitle
\endgroup
\begin{abstract}
We study the asymptotic relations between certain singular and constrained control problems for one-dimensional diffusions with both discounted and ergodic objectives. By constrained control problems we mean that controlling is allowed only at independent Poisson arrival times. We show that when the underlying diffusion is recurrent, the solutions of the discounted problems converge in Abelian sense to those of their ergodic counterparts. Moreover, we show that the solutions of the constrained problems converge to those of their singular counterparts when the Poisson rate tends to infinity. We illustrate the results with drifted Brownian motion and Ornstein-Uhlenbeck process.
\end{abstract}


\vspace{10pt}

\noindent \textup{2010} \textit{Mathematics Subject Classification}: \textup{49N25, 49L20, 60J60}

\section{Introduction}

This paper is concerned with asymptotic relations between certain discounted and ergodic control problems for one-dimensional diffusions. More precisely, the following control problems are considered:
\begin{itemize}
    \item[(A)] Classical singular stochastic control problems with both discounted and ergodic criteria 
    \item[(B)] Constrained bounded variation control problems where controlling is allowed only at the independent Poisson arrival times with both discounted and ergodic criteria 
\end{itemize}
These control problems are expected to be linked to each other via certain limiting properties. For instance, it is often expected that in item (A), the values of the problems with discounted criterion are connected to the ergodic problems in an Abelian sense, when the discounting factor vanishes. This relationship, often called the vanishing discount method and sometimes used in a heuristic manner, can be used to solve the ergodic problems \cite{FlowSystems2005, Robin1983, Gatarek1990, Borkar}.

Regarding item (B), the problems of this form have attracted attention in the recent years \cite{Lempa2014, LempaSaarinen2019, Wang2001, MenaldiRobin2017, MenaldiRobin2018}, for related studies in optimal stopping see \cite{Lempa2012, DupuisWang2002, Guo2005}. In these problems, it is reasonable to expect that the value functions of the constrained problems should converge to the values of their singular counterparts as the Poisson arrival rate of the control opportunities tends to infinity.

The main contribution of this paper is that we prove these expectations to be correct for time-homogeneous control problems with one-dimensional diffusion dynamics; our findings are summarized in Figure 1. These diffusion models are important in many applications. Furthermore, the time-homogeneous structure allows explicit calculations, by which we can first solve the HJB-equations of both discounted and ergodic problems separately and then establish that the solutions satisfy the desired limiting properties. This is in contrast to the vanishing discount method, where the HJB-equation of the ergodic problem is solved using the solution of the discounted problem \cite{Robin1983}.  

The remainder of the paper is organized as follows. In section 2, we set up the diffusion dynamics. In section 3, we introduce the control problems and study the functionals appearing in their analysis. The asymptotic relations are proved in section 4. Paper is concluded with an explicit examples in section 5.

\section{Underlying dynamics}
Let $(\Omega, \mathcal{F}, \{\mathcal{F}_t\}_{t \geq 0}, \mathbb{P})$ be a filtered probability space which satisfies the usual conditions. We consider an uncontrolled real-valued process $X$ defined on $(\Omega, \mathcal{F}, \{\mathcal{F}_t\}_{t \geq 0}, \mathbb{P})$ which is modelled as a strong solution to the Itô stochastic differential equation
\begin{equation}
dX_t=\mu(X_t)dt+ \sigma(X_t)dW_t, \quad \quad X_0 = x,
\end{equation}
where $W_t$ is the Wiener process and the functions $\mu$ and $\sigma$ are well-behaved (see \cite{KaratzasShreve1991} chapter 5). For notional convenience, we consider the case where the process evolves in $\mathbb{R}_+$, even though all the results remain unchanged even if the state space would be replaced with any intervalin $\mathbb{R}$.

We define the second-order linear differential operator  $\mathcal{A}$, which represents the infinitesimal generator of the diffusion $X$, as
\begin{equation}
\mathcal{A} = \mu(x) \frac{d}{dx} + \frac{1}{2} \sigma^2(x)\frac{d^2}{dx^2}
\end{equation}
and for a given $r > 0$, we respectively denote the increasing and decreasing solutions to the differential equation $(\mathcal{A}-r)f=0$ by $\psi_r > 0$ and $\varphi_r > 0$. These solutions are often called the \emph{fundamental solutions} and can be identified as the minimal $r$-excessive functions of the diffusion $X$ (see \cite{BorodinSalminen} p. 19).

We denote by $\tau$ the first exist time from $\mathbb{R}_+$, i.e. $\tau = \inf \{ t \geq 0 \mid X_t \not \in \mathbb{R}_+  \}$ and define a set $\mathcal{L}_1^r$ of functions $f$ that satisfy the integrability condition $ \mathbb{E}_x [ \int_0^{\tau}e^{-r s} |f(X_s)| ds ] < \infty $. Using this notation, we define the inverse of the differential operator $(r-\mathcal{A})$, called the resolvent $R_r$, by
\begin{equation*}
(R_rf)(x)=\mathbb{E}_x \bigg[ \int_0^{\tau}e^{-r s} f(X_s) ds \bigg]
\end{equation*}
for all $x \in \mathbb{R}_+$ and $f \in \mathcal{L}_1^r$. We also define the scale density of the diffusion by
\begin{equation*}
S'(x) = \exp \bigg( - \int^x \frac{2 \mu(z)}{\sigma^2(z)}dz \bigg),
\end{equation*}
which is the monotonic (and non-constant) solution to the differential equation $\mathcal{A} f=0$.

Often in computations it is useful to use the formula 
\begin{equation}
\begin{aligned} \label{resolventin laskukaava}
    (R_r f)(x) = & B_r^{-1} \varphi_r(x)\int_0^{x} \psi_r(y)f(y)m'(y)dy \\
      + & B_r^{-1} \psi_r(x)\int_x^{\infty} \varphi_r(y) f(y) m'(y) dy, 
\end{aligned}
\end{equation}
which connects the resolvent and the fundamental solutions $\psi_r$ and $\varphi_r$ (see \cite{BorodinSalminen} p. 19). Here the positive constant, which does not depend on $x$,
$$
B_r = \frac{\psi_r'(x)}{S'(x)} \varphi_r(x) -\frac{\varphi_r'(x)}{S'(x)} \psi_r(x)
$$
is the Wronskian of the fundamental solutions and 
$$
m'(x) = \frac{2}{\sigma^2(x)S'(x)}
$$
denotes the density of the speed measure. We also recall the resolvent equation (see \cite{BorodinSalminen} p. 4)
\begin{equation} \label{ResolventEquation} 
R_q R_r= \frac{R_r-R_q}{q-r}.
\end{equation}

\section{The control problems}

To pose the assumptions for the control problems, we define the function $\theta_r: \mathbb{R}_+ \to \mathbb{R}$ as $\theta_r(x)=\pi(x)+\gamma\rho(x)$, where $\gamma$ is a positive constant and $\pi: \mathbb{R}_+ \to \mathbb{R}$ is the cost function. Here, the function $\rho: \mathbb{R}_+ \to \mathbb{R}$ is defined as $\rho(x)=\mu(x)-rx$, where $r$ is a positive constant called the \emph{discounting factor}. In economic literature, the function $\theta_r$ can be understood as the net convenience yield of holding inventories \cite{Alvarez2004, Dixit1984}. This function appears in wide range of control problems of one-dimensional diffusions, when the criteria to be minimized includes discounting \cite{Lempa2014, Alvarez2004, Matomaki2012}. 

In addition, we note that in the absence of discounting, the function $\theta_r$ reduces to
$$\pi_{\mu}(x) = \pi(x) + \gamma \mu(x),$$
which is in key role in many ergodic control problems of one-dimensional diffusions \cite{LempaSaarinen2019, Alvarez2019}.

We study the control problems under the following assumptionss.

\begin{assumption} We assume that:
\begin{enumerate}
    \item the upper boundary $\infty$ and the lower boundary $0$ are natural,
    \item the cost function $\pi$ is continuous, non-negative and non-decreasing,
    \item the functions $\theta_r$ and id$: x \mapsto x$ are in $\mathcal{L}_1^r$,
    \item there is a unique state $x^* \geq 0$ such that $\theta_r$ (and also $\pi_{\mu}$) is decreasing on $(0, x^* )$ and increasing on $(x^* , \infty)$ and that it satisfies the limiting condition ${ \lim_{x \to \infty} \theta_r(x) \geq 0}$. 
\end{enumerate}
\end{assumption}

\begin{assumption}
\begin{enumerate} We assume that
\item  $m(0,y)=\int_{0}^{y} m'(z)dz <\infty$ and $\int_{0}^{y} \pi_{\mu}(z) m'(z)dz<\infty$ for all $y \in \mathbb{R}_+$,
\item $\lim_{x \downarrow 0} S'(x) = \infty$.
\end{enumerate}
\end{assumption}

We make some remarks on these assumptions. First, we assume that the uncontrolled state variable $X$ cannot become infinitely large or zero in finite time, see \cite{BorodinSalminen} pp. 18–20, for a characterization of the boundary behavior of diffusions. Second, the cost function is non-decreasing and non-negative which is in line with usual economic applications. Third, the resolvents $(R_r \theta_r)(x)$ and $(R_r \, \text{id})(x)$ exists. Fourth, we restrict our attention to the case, where the function $\theta_r$ ($\pi_{\mu}$) has a unique global minimum at $x^*$. Finally, the first part in assumption 2 guarantees that the underlying diffusion has a stationary distribution (see \cite{BorodinSalminen} p. 37).

Lastly, we assume that

\begin{assumption}
The process $N_t$ is a Poisson process with a parameter $\lambda$ and it is independent of the underlying diffusion $X_t$. Furthermore, we assume that, the filtration $\{\mathcal{F}_t\}_{t \geq 0}$ is rich enough to carry the Poisson process $N = (N_t,\mathcal{F}_t)_{t \geq 0}$. We denote the jump times of $N_t$ by $T_i$.
\end{assumption}

Before stating the control problems, we define the auxiliary functions ${\pi_{\gamma}: \mathbb{R}_+ \to \mathbb{R}}$ and $g: \mathbb{R}_+ \to \mathbb{R}$ as
\begin{align*}
&g(x)=\gamma x-(R_r\pi)(x), \\
&\pi_{\gamma}(x) = \lambda \gamma x + \pi(x),
\end{align*}
where $\gamma$, $r$ are the same positive constants as in the definition of $\theta_r$ and $\lambda$ is the intensity of the Poisson process in assumption 3. The next lemma gives useful relationships between these auxiliary functions, $\theta_r$ and $\pi$. The lemma can be proved using the resolvent equation (\ref{ResolventEquation}) and the harmonicity property $(A-r)(R_r\pi)(x)+\pi(x)=0$.
\begin{lemma} \label{LinksBetweenFunctions} Assume that $g, \pi_{\gamma}, \theta_r \in \mathcal{L}_1^r$. Then
\begin{align}
& (\mathcal{A}-r)g(x) = \theta_r(x), \label{funktion g ja thetan yhteys} \\
& (R_{r+\lambda} \pi_\gamma)(x) = \lambda (R_{r+\lambda}g)(x) + (R_r\pi)(x), \\
& \lambda (R_{r+\lambda}g)(x) = (R_{r+\lambda}\theta_r)(x)+g(x). \label{funktion g resolventin ja thetan yhteys}
\end{align}
\end{lemma}

The following functionals are in a key role in the main results of the study, as they offer convenient representations of the integral functionals that include $\psi_r$ and $\varphi_r$:
\begin{align*}
 L_{\theta}^r(x) & = r\int_x^{\infty}  \varphi_{r}(y) \theta(y) m'(y)dy + \frac{\varphi'_{r}(x)}{S'(x)}\theta(x), \\
 K_{\theta}^r(x) & = r \int_0^x \psi_r(y) \theta(y) m'(y)dy - \frac{\psi_r'(x)}{S'(x)} \theta(x), \\
  H(x,y) & = \int_x^{y} (\pi_{\mu}(z)-\pi_{\mu}(x))m'(z)dz.
\end{align*}

We now formulate downward singular control problems of one-dimensional diffusions and similar problems where controlling is allowed only at exogenously given Poisson arrival times. This latter problem is called \emph{constrained} control problem. 

We assume in all of the following theorems, that the controlled dynamics are given by the stochastic differential equation
\begin{equation*}
X_t^D = \mu(X_t^D) dt + \sigma(X_t^D)dW_t - \gamma dD_t, \quad X_0^D = x \in \mathbb{R_+},
\end{equation*}
where $D_t$ denotes the applied control policy and $\gamma$ is positive constant (the coefficient $\gamma$ is often interpreted as proportional transaction cost). 

\begin{assumption}[Admissible control policies] \quad \quad 
\begin{enumerate}
    \item In the singular problems (theorems 1 and 2 below), we call a control policy $D^s_t$ admissible, if it is non-negative, non-decreasing, right-continuous, and $\{\mathcal{F}_t\}_{t \geq 0}$-adapted, and denote the set of admissible controls by $\mathcal{D}_s$.
    \item In the constrained problems (theorems 3 and 4 below) the set of admissible controls $\mathcal{D}$ is given by those non-decreasing, left-continuous processes $D_{t}$ that have the representation
    $$
    D_t=\int_{[0,t)} \eta_s dN_s,
    $$
    where $N$ is the signal process and the integrand $\eta$ is $\{\mathcal{F}_t\}_{t \geq 0}$-predictable.
\end{enumerate}
\end{assumption}

We introduce the main results of the control problems below and refer to \cite{Alvarez2001, AlvarezLempa2008, Alvarez2019, AlvarezHening2019, Lempa2014, LempaSaarinen2019} for full discussions.

\begin{theorem}[Singular control with discounted criteria \cite{Alvarez2001} pp.1701-1702, \cite{AlvarezLempa2008} pp.714-715]
Under the assumption 1, the optimal control policy minimizing the objective
\begin{equation*}
     J_s(x,D^s)=\mathbb{E}_x \left[ \int_0^{\infty} e^{-rt}(\pi(X_t^{D^s})dt + \gamma dD^s_t) \right],
\end{equation*}
where $D^s_t \in \mathcal{D}_s$ is
\begin{equation*}
    D^s_t =
    \begin{cases}
    \mathcal{L}(t, y^*_s), & t>0, \\
    (x-y^*_s)^+, & t=0,
    \end{cases}
\end{equation*}
where $\mathcal{L}(t, y^*_s)$ denotes the local time push of the process $X_t$ at the boundary $y^*_s$. The boundary $y^*_s$ is characterized by the unique solution to the equation
\begin{equation} \label{ys yhtalo}
    K_{\theta_r}^r(y^*_s) = 0.
\end{equation}
Moreover, the value of the problem reads as
\begin{equation*}
V_s(x) := \inf_{D^s \in \mathcal{D}^s} J_s(x,D^s)=
    \begin{cases}
     \gamma x+\frac{\theta_r(y^*_s)}{r}, & \quad x \geq y^*_s \\
     (R_r \pi)(x)-\psi_r(x) \frac{(R_r\pi)'(y^*_s)-\gamma}{\psi_r'(y^*_s)}, & \quad  x< y^*_s.
    \end{cases}
\end{equation*}
\end{theorem}

\begin{theorem}[Singular control with ergodic criteria \cite{Alvarez2019} pp.16-17, \cite{AlvarezHening2019} p.7]
Under the assumptions 1 and 2, the optimal control policy minimizing the objective
\begin{equation*}
    J_{se}(x,D^s)=\liminf_{T \to \infty} \frac{1}{T} \mathbb{E}_x \left[ \int_0^{T} (\pi(X_t^{D^s})dt + \gamma dD_t) \right]
\end{equation*}
where $D^s_t \in \mathcal{D}_s$ is
\begin{equation*}
    D^s_t =
    \begin{cases}
    \mathcal{L}(t, b^*_s), & t>0, \\
    (x-b^*_s)^+, & t=0,
    \end{cases}
\end{equation*}
where $\mathcal{L}(t, b^*_s)$ denotes the local time push of the process $X_t$ at the boundary $b^*_s$. The boundary $b^*_s$ is characterized by the unique solution to the equation
\begin{equation} \label{bs yhtalo}
  H(0,b^*_s) = 0.
\end{equation}
Moreover, the long run average cumulative yield reads as
\begin{equation*}
   \beta_s := \inf_{D \in \mathcal{D}^s} J_{se}(x,D) = \pi_{\mu}(b^*_s).
\end{equation*}
\end{theorem}

\begin{theorem}[Control with discounted criteria and constraint \cite{Lempa2014} p.115]
Under the assumption 1, the optimal control policy that minimizes the objective
\begin{equation*}
     J(x,D)=\mathbb{E}_x \left[ \int_0^{\infty} e^{-rt}(\pi(X_t^{D})dt + \gamma dD_t) \right],
\end{equation*}
where $D_t \in \mathcal{D}$, is as follows. If the controlled process $X^{D}$ is above the threshold $y^*$ at a jump time $T_i$ of $N$, i.e. $X^{D}_{T_{i-}} > y^*$ for any $i$, the decision maker should take the controlled process $X^{D}$ to $y^*$. Further, the threshold $y^*$ is uniquely determined by
\begin{equation*}
\psi_r'(y^*) L^{r+\lambda}_g(y^*)=g'(y^*)L^{\lambda}_{\psi_r}(y^*),
\end{equation*}
which can be rewritten as
\begin{equation} \label{condition rewritten constraint discount}
  \psi'_r(y^*)L_{\theta_r}^{r+\lambda}(y^*) = - \varphi_{r+\lambda}'(y^*)K_{\theta_r}^r(y^*).
\end{equation}
In addition, the value function $V(x):= \inf_{D \in \mathcal{D}} J(x,D)$ of the problem reads as
\begin{equation} \label{value of discounted constraint}
V(x)=
    \begin{cases}
     \gamma x+ (R_{r+\lambda}\theta_r)(x)-\frac{(R_{r+\lambda}\theta_r)'(y^*)}{\varphi_{r+\lambda}'(y^*)}\varphi_{r+\lambda}(x)+ A(y^*), & \quad x \geq y^* \\ \gamma x+
      (R_r \theta_r)(x)-\psi_r(x) \frac{(R_r \theta_r)'(y^*)}{\psi_r'(y^*)}, & \quad  x< y^*
    \end{cases}
\end{equation}
where
\begin{equation*}
   A(y^*) =  \frac{\lambda}{r} \bigg[ (R_{r+\lambda}\theta_r)(y^*) - (R_{r+\lambda}\theta_r)'(y^*)\frac{\varphi_{r+\lambda}(y^*)}{\varphi_{r+\lambda}'(y^*)} \bigg].
\end{equation*}
\end{theorem}
\begin{proof}
We only prove that the optimality condition can be rewritten as (\ref{condition rewritten constraint discount}), and refer to \cite{Lempa2014} for the rest of the claim. To prove the representation, we first use the lemma \ref{lemma L and K second order}, and then the formulas (\ref{funktion g ja thetan yhteys}) and (\ref{funktion g resolventin ja thetan yhteys}), to get
\begin{align*}
    & \frac{2\lambda S'(y^*)}{\sigma^2(y^*)} \Big[ \psi_r'(y^*) L_g(y^*)-g'(y^*)L_{\psi_r}(y^*) \Big] \\
    & = \psi'_r(y^*)(\varphi_{r+\lambda}''(y^*) \lambda (R_{r+\lambda}g)'(y^*)-\varphi_{r+\lambda}'(y^*)\lambda (R_{r+\lambda}g)''(y^*)) \\
    & - g'(y^*)(\varphi_{r+\lambda}''(y^*)\psi_r'(y^*)-\varphi_{r+\lambda}'(y^*)\psi_r''(y^*)) \\
    & = \psi_r'(y^*)(\varphi_{r+\lambda}''(y^*)(R_{r+\lambda} \theta_r)'(y^*)- \varphi_{r+\lambda}'(y^*) (R_{r+\lambda} \theta_r)''(y^*)) \\
    & - \varphi_{r+\lambda}'(y^*)(\psi_r''(y^*)(R_r\theta_r)'(y^*)-\psi_r'(y^*)(R_r\theta_r)''(y^*)).
\end{align*}
Utilizing the lemma \ref{lemma L and K second order} again, we see that the optimality condition has the form 
\begin{equation*} 
  \psi'_r(y^*)L_{\theta_r}^{r+\lambda}(y^*) = - \varphi_{r+\lambda}'(y^*)K_{\theta_r}^r(y^*).
\end{equation*}
\end{proof}

\begin{theorem}[Control with ergodic criteria and constraint \cite{LempaSaarinen2019} p.16]
Under the assumptions 1, 2 and that $\pi(x) \geq C(x^{\alpha}-1)$, where $\alpha$ and $C$ are positive constants, the optimal control policy minimizing the objective
\begin{equation*}
    J_e(x,D)=\liminf_{T \to \infty} \frac{1}{T} \mathbb{E}_x \left[ \int_0^{T} (\pi(X_t^D)dt + \gamma dD_t) \right],
\end{equation*}
where $D_t \in \mathcal{D}$, is as follows. If the controlled process $X^{D}$ is above the threshold $b^*$ at a jump time $T_i$ of $N$, i.e. $X^{D}_{T_{i-}} > b^*$ for any $i$, the decision maker should take the controlled process $X^{D}$ to $b^*$. Further, the threshold $b^*$ is uniquely determined by
\begin{equation} \label{b yhtalo}
   S'(b^*) m(0,b^*) L_{\pi_{\mu}}^{\lambda}(b^*) = -\varphi_{\lambda}'(b^*)H(0,b^*),
\end{equation}
and the long run average cumulative yield $\beta$ reads as
\begin{equation*}
    \beta := \inf_{D \in \mathcal{D}} J_e(x,D) = m(0,b^*)^{-1} \left[ \int^{b^*}_0 \pi_{\mu}(z)m'(z) dz \right].
\end{equation*}
\end{theorem}

\begin{remark}
The boundary classifications for the underlying diffusion can be relaxed in all of the above theorems. For example, in theorem 3 it can be shown that the results stays unchanged when the lower boundary is exit or killing, see \cite{Lempa2014} p. 5.
\end{remark}

The optimal policies in the above theorems can be summarised as follows. In the singular control problems the optimal policies are local time type barrier policies. In other words, when the process is below some constant boundary $y^*_s$ the process should be left uncontrolled, but it should never be allowed to cross it, i.e. it is reflected at $y^*_s$. The situation in the problems with constraint is similar: when the process is below some threshold $y^*$ we do not act, but if the process crosses the boundary, and the Poisson process jumps, we immediately push it down to $y^*$ and start it anew. In other words, the optimal strategy in all of the problems is to exert control at the ‘maximum rate’ when the process is at (or above) the corresponding boundary.

\section{Main results}

In the next lemma, we collect useful representations for functionals $K^f_r$ and $L_f^r$.
\begin{lemma} \label{lemma L and K second order}
The functions $L_f$ and $K_f$ have alternative representations
\begin{align*} 
    L_f^r(x) = & \frac{\sigma^2(x)}{2  S'(x)}[\varphi_{r}''(x) (R_{r}f)'(x)-\varphi_{r}'(x)  (R_{r}f)''(x)], \\
    K_f^r(x) = & \frac{\sigma^2(x)}{2 S'(x)} [\psi'_{r}(x) (R_{r} f)''(x)- \psi_{r}''(x)  (R_{r} f)'(x)].
\end{align*}
\end{lemma}
\begin{proof}
The proof for the claim on $L_f$ is in \cite{Lempa2014} lemma 2 and the proof on $K_f$ is completely analogous.
\end{proof}
Under our assumption that the boundaries are natural, we have that
\begin{align} \label{luonnolliset reunat integraali}
    \frac{\varphi'_{r}(x)}{S'(x)}  = -r \int_x^{\infty} \varphi_r(y)m'(y)dy, \quad
    \frac{\psi'_{r}(x)}{S'(x)}  = r \int_0^{x} \psi_r(y)m'(y)dy,
\end{align}
and thus, we can further rewrite 
\begin{align*}
 L_f^r(x) & = r\int_x^{\infty}  \varphi_{r}(y) (f(y)-f(x)) m'(y)dy, \\
 K_f^r(x) & = r \int_0^x \psi_r(y) (f(y)-f(x)) m'(y)dy.
\end{align*}



In the next proposition we prove that the just introduced functionals, often appearing in bounded variation control problems of one-dimensional diffusion processes, satisfy asymptotic properties that are needed to establish the relationships between different control problems.
\begin{proposition} \label{propositio asymptotiikasta} 
Under the assumption 1, we have the following limiting properties
\begin{align*}
     \frac{L_{\theta_r}^{r+\lambda}(x)}{(r+\lambda)\varphi_{r+\lambda}(x)}  \xrightarrow{ \lambda \to \infty } 0, \quad
     \frac{K_{\theta_r}^{r+\lambda}(x)}{(r+\lambda)\psi_{r+\lambda}(x)}  \xrightarrow{ \lambda \to \infty } 0.
\end{align*}
In addition, if the underlying diffusion $X_t$ is recurrent, i.e. $\mathbb{P}_x[\tau_z < \infty] = 1$ for all $x,z \in \mathbb{R}_+$,  then
\begin{align*}
     \frac{L_{\theta_r}^r(x)}{r\varphi_r(x)}  \xrightarrow{ r \to 0 } H(x,\infty), \quad 
     \frac{K_{\theta_r}^r(x)}{r\psi_r(x)}  \xrightarrow{ r \to 0 } H(0,x),
\end{align*}
where 
\begin{equation*}
    H(x,y) = \int_x^{y} (\pi_{\mu}(z)-\pi_{\mu}(x))m'(z)dz.
\end{equation*}
\end{proposition}
\begin{proof}
Let $\tau_z = \inf\{ t \geq 0 \mid X_t = z \}$. Then for all $s > 0$ we have 
\begin{equation} \label{pysaytysaikaesitys fraction}
\mathbb{E}_x[e^{-s\tau_z} \mid \tau_z < \infty] =
    \begin{cases}
    & \dfrac{\psi_{s}(x)}{\psi_{s}(z)}, \quad x \leq z \\
    & \dfrac{\varphi_{s}(x)}{\varphi_{s}(z)}, \quad x > z.
    \end{cases}
\end{equation}
Therefore, by letting $s \to 0+$ we get by monotone convergence that
\begin{equation}
    \begin{aligned} \label{osamaara limit nollassa}
   & \lim_{s \to 0+} \frac{\psi_{s}(x)}{\psi_{s}(z)} = \mathbb{P}_x[\tau_z < \infty] = 1, \\
   & \lim_{s \to 0+} \frac{\varphi_{s}(x)}{\varphi_{s}(z)} = \mathbb{P}_x[\tau_z < \infty] = 1,
    \end{aligned}
\end{equation}
under the assumption that the underlying diffusion is recurrent. In addition, again by (\ref{pysaytysaikaesitys fraction}), we find that
\begin{align} \label{limit infinity of fraction}
   \lim_{s \to \infty} \frac{\psi_{s}(x)}{\psi_{s}(z)}  =  0, \quad \, \, \lim_{s \to \infty} \frac{\varphi_{s}(x)}{\varphi_{s}(z)} = 0.
\end{align}
Since $\lim_{r \to 0+} \theta_r(x) = \pi_{\mu}(x)$, we see by using the above observations that by monotone convergence
\begin{align*}
\frac{L_{\theta_r}^r(x)}{r\varphi_r(x)} & = \int_{x}^{\infty} \frac{\varphi_{r}(z)}{\varphi_r(x)}(\theta_r(z)- \theta_r(x) )m'(z)dz  \to H(x,\infty)  \text{ as }  r \to 0, \\
\frac{K_{\theta_r}^r(x)}{r\psi_r(x)} & = \int_{0}^{x} \frac{\psi_{r}(z)}{\psi_r(x)}(\theta_r(z)- \theta_r(x) )m'(z)dz \to H(0,x) \text{ as }  r \to 0.
\end{align*}
Similarly, by utilizing (\ref{limit infinity of fraction}) we obtain
\begin{align*}
\frac{L_{\theta_r}^{r+\lambda}(x)}{(r+\lambda)\varphi_{r+\lambda}(x)} & =  \int_{x}^{\infty} \frac{\varphi_{r+\lambda}(z)}{\varphi_{r+\lambda}(x)}(\theta_r(z)- \theta_r(x) )m'(z)dz  \to 0  \text{ as } \lambda \to \infty, \\
\frac{K_{\theta_r}^{r+\lambda}(x)}{(r+\lambda)\psi_{r+\lambda}(x)} & = \int_{0}^{x} \frac{\psi_{r+\lambda}(z)}{\psi_{r+\lambda}(x)}(\theta_r(z)- \theta_r(x) )m'(z)dz  \to 0  \text{ as }  \lambda \to \infty. 
\end{align*}
\end{proof}

\begin{proposition}[Asymptotics of the optimal thresholds] \label{raja-arvot reunoista}
Under the assumptions of proposition 1 the optimal thresholds satisfy the following asymptotic results in terms of the intensity of the Poisson process
\begin{align*}
     y^* \xrightarrow{ \lambda \to \infty }  y^*_s, \quad
     b^* \xrightarrow{ \lambda \to \infty }  b^*_s,
\end{align*}
and if the underlying diffusion is also recurrent, we have the vanishing discount factor limits
\begin{align*}
     y^*_s \xrightarrow{ r \to 0 }  b^*_s, \quad
     y^* \xrightarrow{ r \to 0 }  b^*.
\end{align*}
\end{proposition}

\begin{proof}
We prove the first and the last claim, as the proof of the second claim can be found in \cite{LempaSaarinen2019}, proposition 3, and the proof of the third claim in \cite{AlvarezHening2019}, lemma 3.1. Define the functions
\begin{align*}
    G(x) & = \psi'_r(x)L_{\theta_r}^{r+\lambda}(x) + \varphi_{r+\lambda}'(x)K_{\theta_r}^r(x), \\
    F(x) & = S'(x) m(0,x) L_{\pi_{\mu}}^{\lambda}(x) + \varphi_{\lambda}'(x)H(0,x),
\end{align*}
and let $y^*_s$, $y^*$, $b^*$  be such that $K(y^*_s)=0$, $G(y^*)=0$ and $F(b^*)=0$. Using these notations, the first claim can be re-expressed as
\begin{align*}
     \frac{G(y^*_s)}{\varphi_{r+\lambda}(y^*_s)} \to  0 \text{ as } \lambda \to \infty.
\end{align*}
Now, utilizing the condition $K(y^*_s)=0$, together with the lemma \ref{propositio asymptotiikasta}, we have that
\begin{equation*}
   \frac{L_{\theta_r}^{r+\lambda}(y^*_s)}{\varphi_{r+\lambda}(y^*_s)} \to 0  \text{ as } \lambda \to \infty.
\end{equation*}
To prove the last claim, we first note, as above, that the claim is equivalent to
\begin{align*}
     \frac{G(b^*)}{\psi_r'(b^*)} \to  0 \text{ as } r \to 0.
\end{align*}
Hence, utilizing (\ref{luonnolliset reunat integraali}), we get that 
\begin{align*}
     \frac{G(b^*)}{\psi_r'(y^*_s)} &= \lambda \int_{b^*}^{\infty} \varphi_{r+\lambda}(z)\theta_r(z)m'(z)dz \\  &+ \frac{\varphi_{r+\lambda}'(b^*)}{S'(b^*)} \frac{\int_0^{b^*}\psi_r(z)\theta_r(z)m'(z)dz}{\int_0^{b^*}\psi_r(z) m'(z)dz}.
\end{align*}
By lemma \ref{propositio asymptotiikasta} we have that
\begin{align*}
     & \frac{\int_0^{b^*}\psi_r(z)\theta_r(z)m'(z)dz}{\int_0^{b^*}\psi_r(z) m'(z)dz} = \frac{K_{\theta_r}^r(b^*) + \theta_r(b^*)\int_0^{b^*}\psi_r(z)m'(z)dz}{\int_0^{b^*}\psi_r(z) m'(z)dz} \\
     & = \frac{\frac{K_{\theta_r}^r(b^*)}{\psi_r(b^*)}}{\int_0^{b^*}\frac{\psi_r(z)}{\psi_r(b^*)} m'(z)dz} + \theta_r(b^*) \\
     & \xrightarrow{ r \to 0 } \frac{H(0,b^*)}{m(0,b^*)} + \pi_{\mu}(b^*).
\end{align*}
Thus, the claim follows from the assumption $F(b^*)=0$.
\end{proof}

Similar limiting results hold also for the corresponding values of the defined control problems. However, it is clear that in terms of the vanishing discounting factor, the results hold only in the following Abelian sense.

\begin{proposition}[Asymptotics of the values] \label{raja-arvot arvoista} Under the assumptions of proposition 1 the values of the control problems satisfy the following asymptotic results
\begin{align*}
     V(x) \xrightarrow{ \lambda \to \infty} V_s(x), \quad \,
     \beta \xrightarrow{ \lambda \to \infty } \beta_s.
\end{align*}
Also, if the underlying diffusion is recurrent, we have the following Abelian limits
\begin{align*}
    r V(x) \xrightarrow{ r \to 0 } \beta, \quad \quad  r V_s(x) \xrightarrow{ r \to 0 } \beta_s.
\end{align*}
\end{proposition}
\begin{proof}
For the last claim see lemma 3.1 of \cite{AlvarezHening2019}. To prove the third, we first re-write the value function (\ref{value of discounted constraint}) using lemma \ref{LinksBetweenFunctions} as
\begin{equation} 
V(x)=
    \begin{cases} \label{value of constraint rewritten}
     \gamma x+ (R_{r+\lambda}\theta_r)(x)-\frac{(R_{r+\lambda}\theta_r)'(y^*)}{\varphi_{r+\lambda}'(y^*)}\varphi_{r+\lambda}(x)+ A(y^*), & \quad x \geq y^*, \\
      \gamma x +  \psi_r(x)\bigg[\frac{(R_r \theta_r)(x)}{\psi_r(x)}-\frac{(R_r \theta_r)'(y^*)}{\psi_r'(y^*)} \bigg], & \quad  x< y^*,
    \end{cases}
\end{equation}
where
\begin{equation*}
   A(y^*) =  \frac{\lambda}{r} \bigg[ (R_{r+\lambda}\theta_r)(y^*) - (R_{r+\lambda}\theta_r)'(y^*)\frac{\varphi_{r+\lambda}(y^*)}{\varphi_{r+\lambda}'(y^*)} \bigg].
\end{equation*}
We notice that when $x > y^*$, the value function $r V(x)$ has convenient presentation in terms of the limit $r \to 0$. However, when $x < y^*$ we have to proceed as follows. Because $V(x)$ is continuous across the boundary $y^*$, we find 
\begin{align} \label{constraint ongelman jatkuvuusehto}
    & (r+\lambda)\psi_r'(y^*)(\varphi_{r+\lambda}'(y^*)(R_{r+\lambda}\theta_r)(y^*) - \varphi_{r+\lambda}(y^*)(R_{r+\lambda}\theta_r)'(y^*)) \\ 
    = & r\varphi_{r+\lambda}'(y^*)(\psi_r'(y^*)(R_{r}\theta_r)(y^*) - \psi_r(y^*)(R_{r}\theta_r)'(y^*)), \nonumber
\end{align}
which can be re-organized as
\begin{align} \label{constraint ongelman jatkuvuusehto yli reunan}
    & - r \frac{(R_{r}\theta_r)'(y^*)}{\psi_r'(y^*)} + r \frac{(R_{r}\theta_r)(y^*)}{\psi_r(y^*)}  \\
    & = (r+\lambda)\bigg(\frac{\varphi_{r+\lambda}'(y^*)(R_{r+\lambda}\theta_r)(y^*) - \varphi_{r+\lambda}(y^*)(R_{r+\lambda}\theta_r)'(y^*)}{\varphi_{r+\lambda}'(y^*) \psi_r(y^*)} \bigg). \nonumber
\end{align}
Thus, we get that
\begin{align*}
    & r\gamma x +  r \psi_r(x)\bigg[\frac{(R_r \theta_r)(x)}{\psi_r(x)}-\frac{(R_r \theta_r)'(y^*)}{\psi_r'(y^*)} \bigg] \\
    & = r \gamma x +  r(R_r \theta_r)(x) - r  \frac{\psi_r(x)}{\psi_r(y^*)} (R_{r}\theta_r)(y^*)  \\
    & + (r+\lambda)\frac{\psi_r(x)}{\psi_r(y^*)}\bigg(\frac{\varphi_{r+\lambda}'(y^*)(R_{r+\lambda}\theta_r)(y^*) - \varphi_{r+\lambda}(y^*)(R_{r+\lambda}\theta_r)'(y^*)}{\varphi_{r+\lambda}'(y^*)} \bigg). 
\end{align*}
Using the formula (\ref{resolventin laskukaava}) and (\ref{osamaara limit nollassa}), we see that 
\begin{align*}
      r (R_r \theta_r)(x)
      & =  r \frac{\varphi_r(x) \int_0^{x} \psi_r(z)\theta_r(z)m'(z)dz +\psi_r(x)\int_x^{\infty} \varphi_r(z) \theta_r(z) m'(z) dz}{\frac{1}{S'(x)}[\psi'_r(x) \varphi_r(x)-\psi_r(x)\varphi_r'(x)]} \\
      & = \frac{\varphi_r(x)\int_0^{x} \psi_r(z)\theta_r(z)m'(z)dz + \psi_r(x) \int_x^{\infty} \varphi_r(z) \theta_r(z) m'(z) dz}{ \varphi_r(x)\int_0^x \psi_r(z)m'(z)dz +\psi_r(x) \int_x^{\infty} \varphi_r(z)m'(z)dz}
      \\
      & =  \frac{\int_0^{x} \frac{\psi_r(z)}{\psi_r(x)}\theta_r(z)m'(z)dz +  \int_x^{\infty} \frac{\varphi_r(z)}{\varphi_r(x)} \theta_r(z) m'(z) dz}{ \int_0^x \frac{\psi_r(z)}{\psi_r(x)}m'(z)dz + \int_x^{\infty} \frac{\varphi_r(z)}{\varphi_r(x)}m'(z)dz}
      \\
      & \xrightarrow{ r \to 0 } \frac{\int_0^{\infty}\pi_{\mu}(z)m'(z)dz}{ \int_0^{\infty} m'(z)dz},
\end{align*}
and thus, by (\ref{osamaara limit nollassa}) we have
\begin{equation*}
     r(R_r \theta_r)(x) - r  \frac{\psi_r(x)}{\psi_r(y^*)} (R_{r}\theta_r)(y^*) \xrightarrow{ r \to 0 } 0.
\end{equation*}
Therefore, by continuity and (\ref{raja-arvot reunoista}), the value function satisfies
\begin{equation*}
rV(x)  \xrightarrow{ r \to 0 }
    \begin{cases}
     \lambda\frac{\varphi_{\lambda}'(b^*)(R_{\lambda} \pi_{\mu})(b^*) - \varphi_{\lambda}(b^*)(R_{\lambda} \pi_{\mu})'(b^*)}{\varphi_{\lambda}'(b^*)}, & \quad x \geq b^* \\
     \lambda\frac{\varphi_{\lambda}'(b^*)(R_{\lambda} \pi_{\mu})(b^*) - \varphi_{\lambda}(b^*)(R_{\lambda} \pi_{\mu})'(b^*)}{\varphi_{\lambda}'(b^*)}, & \quad  x< b^*.
    \end{cases}
\end{equation*}
Finally, utilizing (\ref{resolventin laskukaava}), the limiting value reads as
\begin{equation*}
    - \frac{\lambda S'(b^*)}{\varphi_{\lambda}'(b^*)} \int_{b^*}^{\infty} \pi_{\mu}(z) \varphi_{\lambda}(z) m'(z)dz,
\end{equation*}
which completes the proof of the third claim.

To prove the second claim, we notice that the value function $V(x)$ is independent of $\lambda$ when $x<y^*$. Thus, we focus this time on the region $x>y^*$. We re-organize the terms $V(x)$ in the upper region as
\begin{equation} \label{esitys ylemmassa alueessa arvolle}
     \gamma x+ (R_{r+\lambda}\theta_r)(x)-\frac{(R_{r+\lambda}\theta_r)'(y^*)}{\varphi_{r+\lambda}'(y^*)}\varphi_{r+\lambda}(x)+ A(y^*) 
\end{equation}
where
\begin{equation*}
   A(y^*) =  \frac{\lambda}{r} \bigg[ (R_{r+\lambda}\theta_r)(y^*) - (R_{r+\lambda}\theta_r)'(y^*)\frac{\varphi_{r+\lambda}(y^*)}{\varphi_{r+\lambda}'(y^*)} \bigg].
\end{equation*}
Because diffusions are Feller-processes, we know that $\lambda (R_{r+\lambda}\theta_r) \to \theta_r$ as $\lambda \to \infty$ (in sup-norm), see \cite{RogersWilliams2001} pp. 235. Thus,
\begin{equation*} 
 \gamma x + \frac{\lambda (R_{r+\lambda}\theta_r)(x)}{\lambda} + \frac{\lambda (R_{r+\lambda}\theta_r)(y^*)}{r} \xrightarrow{ \lambda \to \infty } \gamma x + \frac{\theta_r(y^*_s)}{r}.
\end{equation*}
To deal with the remaining terms in (\ref{value of constraint rewritten}), we note that by (\ref{constraint ongelman jatkuvuusehto}) 
\begin{align*}
  \frac{(R_{r+\lambda}\theta_r)'(y^*)}{\varphi_{r+\lambda}'(y^*)} & = \frac{(R_{r+\lambda}\theta_r)(y^*)}{\varphi_{r+\lambda}(y^*)} - \frac{r}{r+\lambda}\frac{(R_{r}\theta_r)(y^*)}{\varphi_{r+\lambda}(y^*)} \\
  & + \frac{r}{r+\lambda} \frac{\psi_r(y^*)}{\psi'_r(y^*)} \frac{(R_{r}\theta_r)'(y^*)}{\varphi_{r+\lambda}(y^*)}.
\end{align*}
Utilizing the above we get by (\ref{osamaara limit nollassa})
\begin{align*}
    &  \frac{(R_{r+\lambda}\theta_r)'(y^*)}{\varphi'_{r+\lambda}(y^*)} \varphi_{r+\lambda}(x) \\
      & =  \frac{\varphi_{r+\lambda}(x)}{\varphi_{r+\lambda}(y^*)} (R_{r+\lambda}\theta_r)(y^*)  - \frac{r}{r+\lambda} \frac{\varphi_{r+\lambda}(x)}{\varphi_{r+\lambda}(y^*)} (R_{r}\theta_r)(y^*) \\
     & + \frac{r}{r+\lambda} \frac{\psi_r (y^*)}{\psi_r'(y^*)} \frac{\varphi_{r+\lambda}(x)}{\varphi_{r+\lambda}(y^*)} (R_{r}\theta_r)'(y^*) \\
      & \xrightarrow{ \lambda \to \infty }  0
\end{align*}
and
\begin{align*}
    & \frac{\lambda}{r} \frac{(R_{r+\lambda}\theta_r)'(y^*)}{\varphi'_{r+\lambda}(y^*)} \varphi_{r+\lambda}(y^*) \\
     & = \frac{\lambda}{r} (R_{r+\lambda}\theta_r)(y^*) - \frac{\lambda}{r+\lambda}(R_{r}\theta_r)(y^*) + \frac{\lambda}{r+\lambda} \frac{\psi_r (y^*)}{\psi_r'(y^*)} (R_{r}\theta_r)'(y^*) \\
     & \xrightarrow{ \lambda \to \infty }  \frac{\theta_r(y^*_s)}{r} - (R_{r}\theta_r)(y^*_s) + \frac{\psi_r(y^*_s)}{\psi_r'(y^*_s)} (R_{r}\theta_r)'(y^*_s).
\end{align*}
As the value function $V_s(x)$ is continuous over the boundary $y^*_s$, we further find that 
$$
\frac{\theta_r(y^*_s)}{r} - (R_{r}\theta_r)(y^*_s) + \frac{\psi_r(y^*_s)}{\psi_r'(y^*_s)} (R_{r}\theta_r)'(y^*_s) = 0.
$$
Combining the above limits the result follows by continuity and proposition \ref{raja-arvot reunoista}.

Lastly, the second claim of the proposition follows by continuity of the functions and proposition \ref{raja-arvot reunoista}, as $\beta_s$ can also be represented as (see \cite{Alvarez2019} pp. 17)
\begin{equation*}
    \beta_s =  m(0,b_s^*)^{-1} \left[ \int^{b_s^*}_0 \pi_{\mu}(z)m'(z) dz \right].
\end{equation*}

\end{proof}

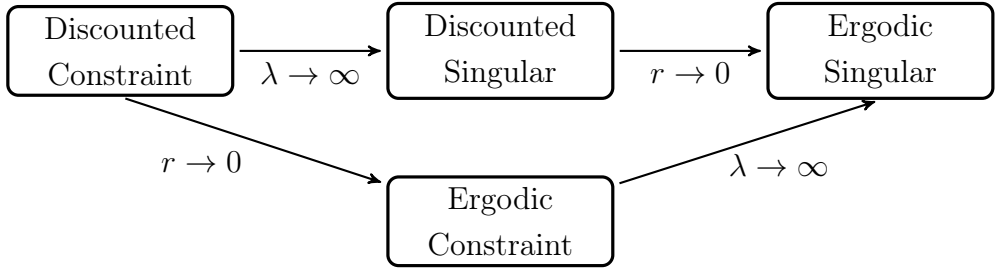
\begin{figure}\label{Fig:Diagram}
\begin{tikzpicture}[node distance=2cm, auto,]
\node[punkt] (dc) {Discounted Constraint};
\node[punkt, right=of dc] (ds) {Discounted Singular}
edge[pil,<-] node[auto] {$\lambda \to \infty$} (dc.east);
\node[punkt, right=of ds] (es) {Ergodic Singular}
edge[pil,<-] node[auto, below] {$r \to 0$} (ds.east);
\node[punkt, below=1cm of ds] (ec) {Ergodic Constraint}
edge[pil] node[auto, below] {$\quad \quad \lambda \to \infty$} (es.south)
edge[pil,<-] node[auto] {$r \to 0$} (dc.south);
\end{tikzpicture}
\caption{Relations between the control problems. These relations hold for the optimal thresholds and also for the values, in the sense of propositions 1 and 2.}
\end{figure}

\section{Illustration}

\subsection{Brownian motion with drift}

Let the underlying process $X_t$ be defined by
\begin{equation}
dX_t = \mu dt + dW_t, \quad \, X_0=x,
\end{equation}
where $\mu > 0$. Also, we let the process evolve in $\mathbb{R}$ and choose a quadratic running cost $\pi(x) = x^2$.
The minimal excessive functions are in this case known to be 
$$
\varphi_{\lambda}(x) = e^{-\big(\sqrt{\mu^2+2 \lambda}+\mu\big) x}, \quad \, \psi_{\lambda}(x) = e^{\big(\sqrt{\mu^2+2 \lambda}-\mu \big) x},
$$ 
and the scale density and speed measure read as 
$$
S'(x)= \exp(-2 \mu x), \quad
m'(x)=2 \exp(2 \mu x),
$$
respectively. The net convenience yield now takes the form $\theta(x) = x^2 + \gamma(\mu-rx)$. We notice immediately that our assumptions 1 and 3 hold and so the results apply in the case of the discounted problems (theorems 1 and 3).

To illustrate the results of proposition \ref{raja-arvot reunoista}, we solve the optimality conditions (\ref{ys yhtalo}) and (\ref{condition rewritten constraint discount}). Conveniently the solution to these equations can be represented explicitly. To solve the equations we need to find the functions $K_{\theta_r}^r(x)$ and $L_{\theta_r}^{r+\lambda}(x)$. Elementary integration yields
\begin{align*}
    K_{\theta_r}^r(x) & = \frac{2 e^{x \alpha_r^+} \big( 2 + (-2 x + r \gamma) \alpha_r^+ \big)}{(\alpha_r^+)^3},\\ 
    L_{\theta_r}^{r+\lambda}(x) & = \frac{2 e^{x \alpha_r^-} \big( 2 + (2 x - r \gamma) \alpha_r^- \big)}{(\alpha_r^-)^3},
\end{align*}
where $\alpha_r^+ = \mu + \sqrt{2r + \mu^2}$ and $\alpha_r^- = \mu - \sqrt{2r + \mu^2}$. Plugging these representations to the equations (\ref{ys yhtalo}) and (\ref{condition rewritten constraint discount}), a simplification yields as solutions the thresholds
\begin{align*}
    & y^*_s = \frac{r \gamma}{2}+ \frac{1}{\alpha_r^+}, 
    & y^* = \frac{r \gamma}{2} + \frac{1}{\alpha_r^+} + \frac{1}{\alpha_{r+\lambda}^-}.
\end{align*}
 Using these explicit representation, we get a limit as in proposition \ref{raja-arvot reunoista}. These thresholds are illustrated in the figure \ref{lplot}.
\begin{figure}
    \centering
    \includegraphics[scale=0.53]{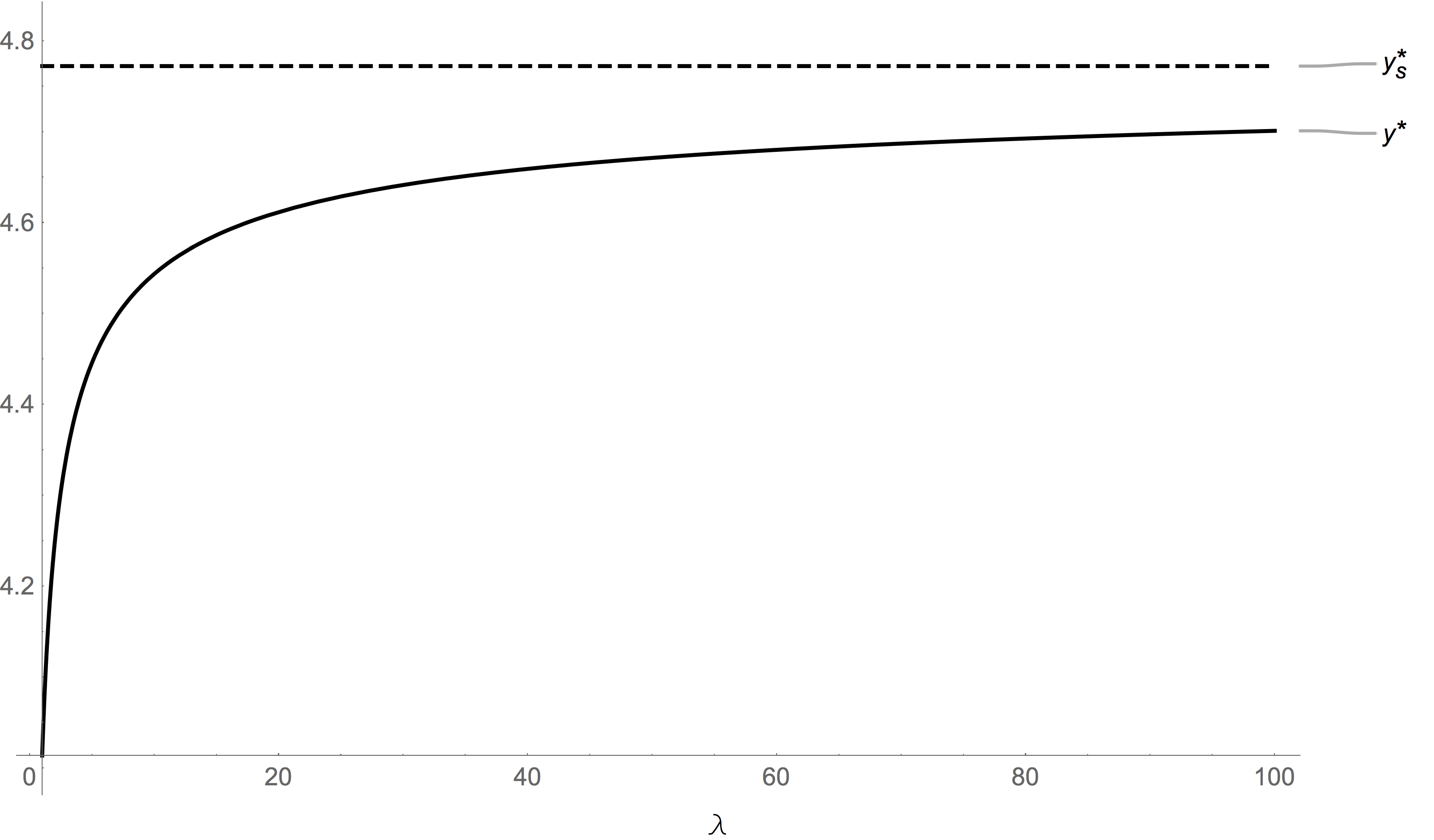}
    \caption{\footnotesize Threshold boundaries as a function of the intensity of the Poisson process with the parameters $\gamma = 0.001$, $\mu=0.1$ and $r=0.001$.}
    \label{lplot}
\end{figure}

\subsection{Controlled Ornstein-Uhlenbeck process}

Consider dynamics that are characterized by a stochastic differential equation 
$$
dX_t = - \beta X_t dt + dW_t, \quad \, X_0=x,
$$
where $\beta > 0$. This diffusion is often used to model continuous time systems that have mean reverting behaviour. To illustrate the results we choose the running cost
$
\pi(x) = |x|,
$
and consequently $\theta(x) = |x|-\gamma (\beta  + r) x$. The scale density and the density of speed measure are in this case
$$
S'(x)= \exp( \beta  x^2) , \quad
m'(x)=2 \exp(- \beta  x^2),
$$
and the minimal $r$-excessive functions read as (see \cite{BorodinSalminen} p. 141)
$$
\varphi_{\lambda}(x) = e^{\frac{\beta x^2}{2}} D_{-\lambda / \beta}(x \sqrt{2 \beta}), \quad \, \psi_{\lambda}(x) = e^{\frac{\beta x^2}{2}} D_{-\lambda / \beta}(-x \sqrt{2 \beta}),
$$ 
where $D_{\nu}(x)$ is a parabolic cylinder function. We note that our assumptions 1 and 2 are satisfied and that this process is positively recurrent recurrent (see \cite{BorodinSalminen} p.141).  Unfortunately, the equations for the optimal thresholds take rather complicated forms, and thus the results in proposition (\ref{raja-arvot reunoista}) are only illustrated numerically.

\begin{figure}
    \centering
    \includegraphics[scale=0.53]{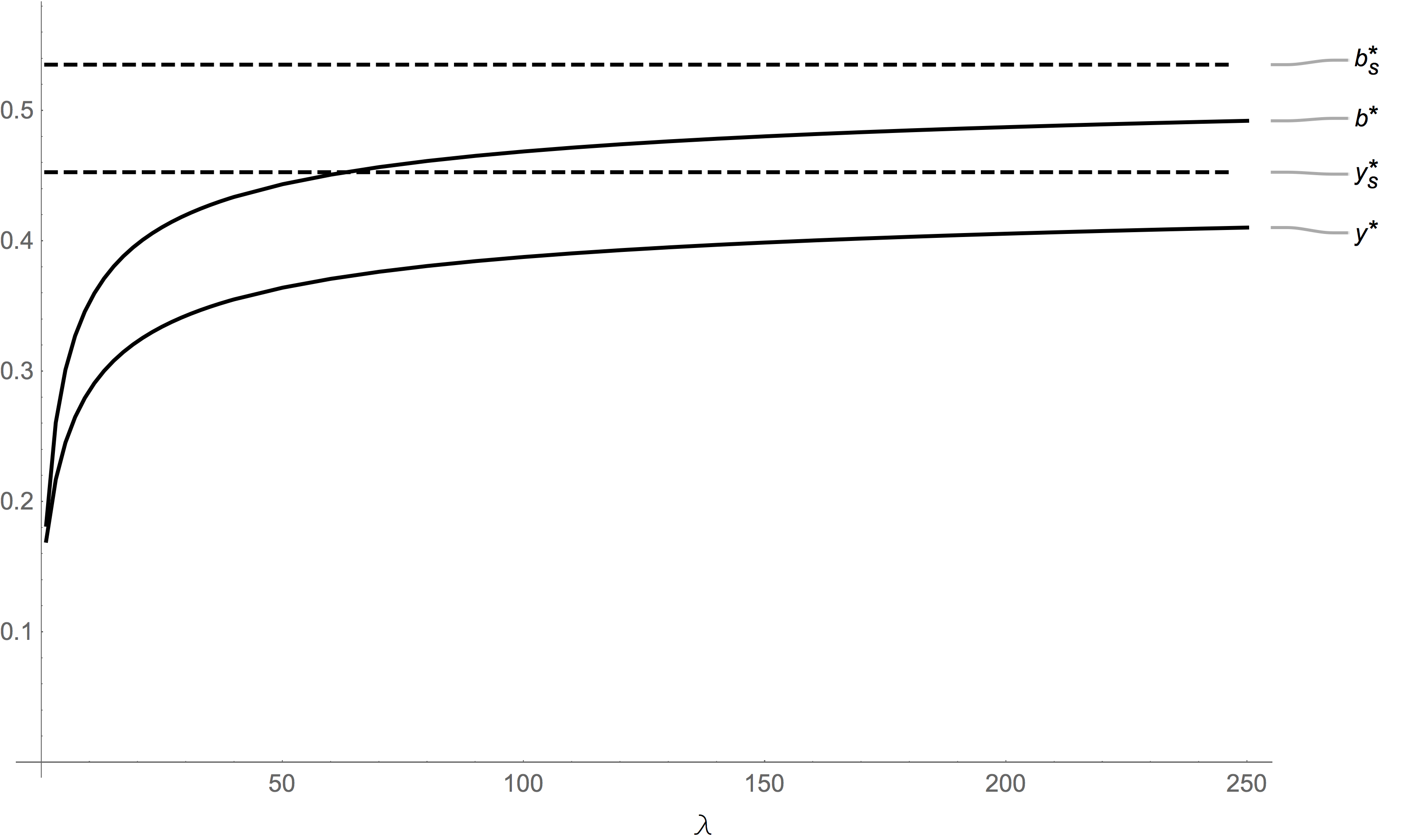}
    \caption{\footnotesize Threshold boundaries as a function of the intensity of the Poisson process with the parameters $\gamma = 0.1$, $\beta=1$ and $r=1$.}
    \label{lOUplot}
\end{figure}

\begin{figure}
    \centering
    \includegraphics[scale=0.53]{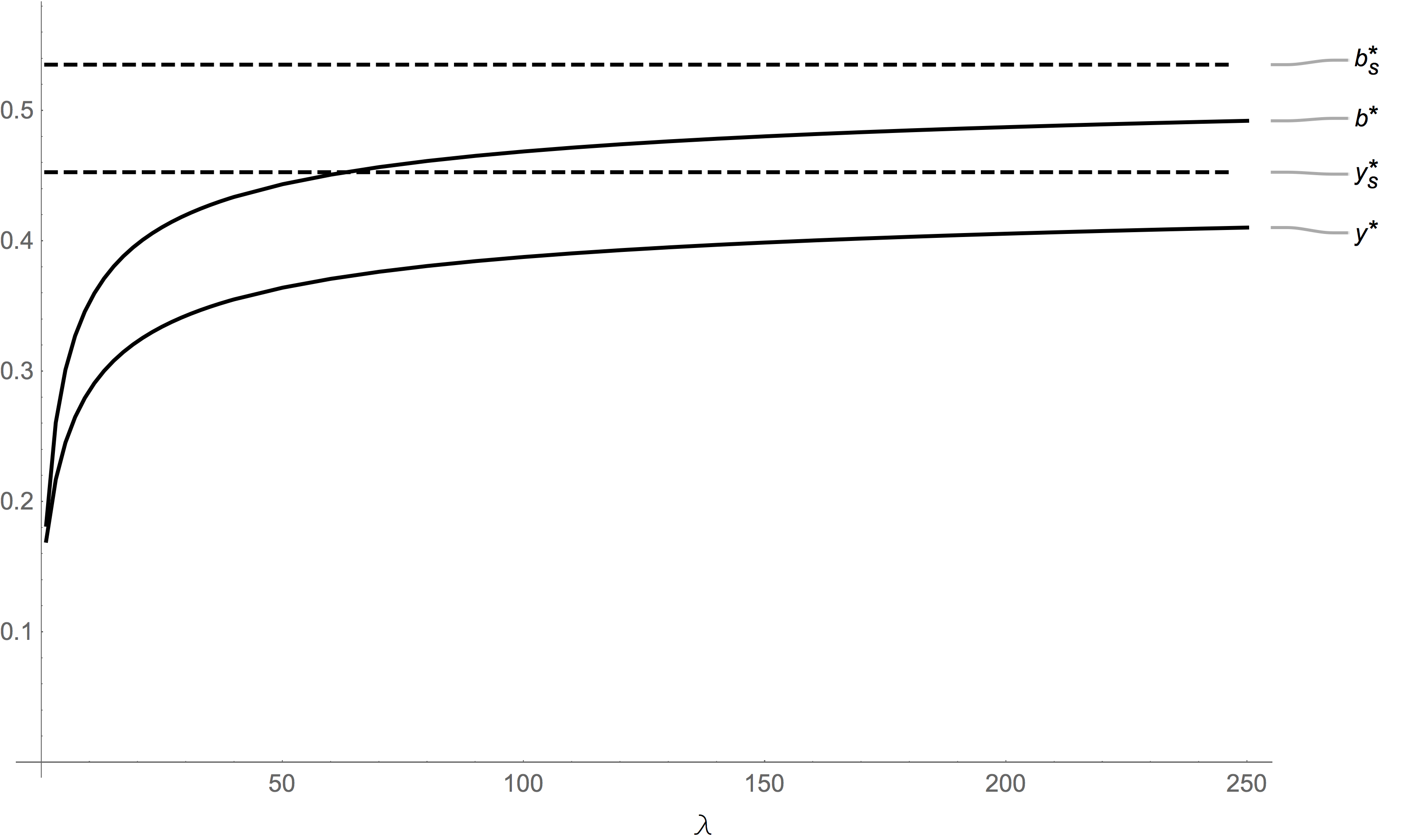}
    \caption{\footnotesize Threshold boundaries as a function of the discounting with the parameters $\gamma = 0.1$, $\beta=1$ and $\lambda=20$.}
    \label{rOUplot}
\end{figure}

\subsection*{Acknowledgements} 
\thispagestyle{empty}

We would like to gratefully acknowledge the emmy.network foundation under the aegis of the Fondation de Luxembourg, for its continued support.

\end{document}